\documentclass[11pt, reqno,a4paper]{amsart}
 \usepackage{amsgen, amstext,amsbsy,amsopn, amsthm, amsfonts,amssymb,amscd,amsmath,euscript,enumerate,url,verbatim,calc,tikz}
\usepackage{hyperref}
\usepackage{MnSymbol}
\usepackage{stmaryrd}
\usepackage{pgfkeys}
\usepackage{mathtools}
\usepackage[left=1.2in,right=1.2in,bottom=1.5in]{geometry}
\usepackage{enumerate}

\usepackage{pst-node}

\usepackage{tikz-cd}

\usepackage{eqnarray,amsmath}

\def\multiset#1#2{\ensuremath{\left(\kern-.3em\left(\genfrac{}{}{0pt}{}{#1}{#2}\right)\kern-.3em\right)}}

\usetikzlibrary{arrows}

 \usepackage{latexsym}
 \usepackage{graphics}
 \usepackage{color}
\usepackage{lastpage}
\usepackage{fancyhdr}
\usepackage{multirow}
\allowdisplaybreaks
\usepackage{graphicx}
\graphicspath{ {F:/IMAGES/} }

\makeatletter
\def\oversortoftilde#1{\mathop{\vbox{\m@th\ialign{##\crcr\noalign{\kern3\p@}%
      \sortoftildefill\crcr\noalign{\kern3\p@\nointerlineskip}%
      $\hfil\displaystyle{#1}\hfil$\crcr}}}\limits}

\def\sortoftildefill{$\m@th \setbox\z@\hbox{$\braceld$}%
  \braceld\leaders\vrule \@height\ht\z@ \@depth\z@\hfill\braceru$}

\makeatother

  \newcommand{\reg}{\operatorname{reg}}

 \newcommand{\ex}{\operatorname{end}}

 \newcommand{\grade}{\operatorname{grade}}
 \newcommand{\depth}{\operatorname{depth}}

  \newcommand{\gr}{\operatorname{gr}}

\newcommand{\proset}{\,\mathrel{\lower 4pt\hbox{$\scriptscriptstyle/$}
\mkern -14mu\subseteq }\,} 

 \newtheorem{theorem}{Theorem}[section]
 \newtheorem{corollary}[theorem]{Corollary}
 \newtheorem{lemma}[theorem]{Lemma}
 \newtheorem{proposition}[theorem]{Proposition}

 \newtheorem{question}[theorem]{Question}

\usepackage{amsmath}
\newtheorem{notation}[theorem]{Notation}
\newtheorem{observation}[theorem]{Observation}

 \theoremstyle{definition}
 
 \newtheorem{remark}[theorem]{Remark}
 \newtheorem{definition}[theorem]{Definition}

 \newtheorem{example}[theorem]{Example}

\newcommand{\n}{\operatorname{n}}

\title[Ratliff-rush filtration, reduction number and postulation number ] {A note on ratliff-rush filtration, reduction number and postulation number of $\mathfrak m$-primary ideals }

\author{Mousumi Mandal and Shruti Priya }

\thanks{AMS Classification 2010: 13H10, 13D40, 13A30.}
\address{Department of Mathematics, Indian Institute of Technology Kharagpur, 721302, India} \email{mousumi@maths.iitkgp.ac.in}
\address{Department of Mathematics, Indian Institute of Technology Kharagpur, 721302, India} \email{shruti96312@kgpian.iitkgp.ac.in}

\begin{document}




\begin{abstract}

     Let $(R,\mathfrak m)$ be a Cohen-Macaulay local ring of dimension $d\geq 2,$ and $I$ an $\mathfrak m$-primary ideal. Let $r(I)$ be the reduction number of $I$, $\n(I)$ the postulation number and $\rho(I)$ the stability index of the Ratliff-Rush filtration with respect to $I$. We prove that for $d=2,$ if $\n(I)=\rho(I)-1,$ then $r(I) \leq \n(I)+2,$ and if $\n(I) \neq \rho(I)-1,$ then $r(I) \geq \n(I)+2.$  
   For $d \geq 3$, assuming $I$ is integrally closed, $\depth \gr(I) = d-2,$ and $\n(I)=-(d-3),$ we prove that $r(I) \geq \n(I)+d$.    Our main result generalizes a result by  Marley on the relation between the Hilbert-Samuel function and the Hilbert-Samuel polynomial by relaxing the condition on the depth of the associated graded ring to the good behaviour of the Ratliff-Rush filtration with respect to $I$ mod a superficial sequence. From this result, it follows that  for  Cohen-Macaulay local rings of dimension $d\geq2$, if $P_{I}(k)=H_{I}(k)$ for some $k \geq \rho(I)$, then  $P_{I}(n)=H_{I}(n)$ for all $n \geq k.$

\end{abstract}

\maketitle

   \section{INTRODUCTION}

 Let $(R, \mathfrak m)$ be a Cohen-Macaulay local ring of dimension $d$  with an infinite residue field, and $I$ be an $\mathfrak m$-primary ideal. Samuel proved that for sufficiently  large $n$, the \textit{Hilbert-Samuel function} $H_{I}(n)= \lambda(R/I^{n})$ (where $\lambda()$ denotes the length) of an $\mathfrak m$-primary ideal $I$ is a polynomial $P_{I}(n)$, known as the \textit{Hilbert-Samuel polynomial.} Since $P_{I}(n)$ takes only numerical values for every integer $n$, by the theory on numerical polynomials, we write

\begin{center}
    $P_{I}(n)=e_{0}(I)\dbinom{n+d-1}{d}-e_{1}(I)\dbinom{n+d-2}{d-1}+\ldots+(-1)^{d}e_{d}(I).$
\end{center}

Here, $e_{0}(I), e_{1}(I), \ldots, e_{d}(I)$ are unique integers known as  the \textit{Hilbert coefficients.} The leading coefficient $e_{0}(I)$ of the polynomial $P_{I}(n)$ is called the \textit{multiplicity} of the ideal $I$ and is denoted by $e(I).$  Hilbert coefficients can be used to characterize the properties of the ring $R$, the ideal $I,$ and the blow-up algebras associated with $I$. 

Marley revealed an interesting connection between the Hilbert coefficients and the reduction number of an ideal $I$. Recall that a reduction of $I$ is an ideal $J \subseteq I$ such that $JI^{n}=I^{n+1}$ for some $n \in \mathbb{N.}$ 
 If $J$ is a reduction of $I$, the \textit{reduction number} of $I$ with respect to $J,$ denoted by  $r_{J}(I),$ is the smallest $n$ such that $JI^{n}=I^{n+1}.$
A reduction is  \textit{minimal} if it is minimal with respect to containment among all reductions. The \textit{reduction number} of $I$ is defined  as
$r(I):=\min\{r_{J}(I): J$ is a minimal reduction of $I\}$. The concept of reduction and minimal reduction was first introduced by Northcott and Rees. Reduction number is useful in studying the Cohen-Macaulay properties of the associated graded ring $\gr(I)= \displaystyle \bigoplus_{n \geq 0} I^{n} / I^{n+1}$ of $R$ with respect to $I.$ 
The reduction number $r(I)$ is said to be independent if $r(I)=r_{J}(I)$ for every minimal reduction $J$ of $I$. In one-dimensional rings, the reduction number $r(I)$ of an $\mathfrak m$-primary ideal $I$ is independent. It was independently proved by Huckaba \cite{hu} and Trung \cite{trung} that if $\depth\gr(I) \geq d-1$, then $r(I)$ is independent. Sally proposed the following question in \cite{sally}: If $(R, \mathfrak m)$ is a Cohen-Macaulay ring of dimension $d$, then is $r(\mathfrak m)$ independent? As a natural extension to this question, one can ask whether $r(I)$ is independent for any $\mathfrak m$-primary ideal $I.$  There are examples that suggest that $r(I)$ is not independent in general. In this paper, we investigate the conditions under which $r(I)$ is independent (see \ref{bb}, \ref{papa}, \ref{purvi}). 

 Ooishi  \cite[Proposition 4.10]{oo} introduced another invariant, known as  the postulation number of $I,$  defined as $\n(I):= \min\{n \in \mathbb{Z} : P_{I}(t)=H_{I}(t)$ for $t > n\}$.  He proved that if $\dim R=1$, then $r(I)=\n(I)+1.$ Marley \cite[Theorem 2.15]{tjm} proved that if $\depth \gr(I)\geq d-1$, then $r(I)=\n(I)+d.$ Much of this paper is inspired by the following question proposed by Marley \cite[Question 2.1]{tjm}: 

\begin{question} \label{csk}
Let $(R, \mathfrak m)$ be a Cohen-Macaulay local ring of dimension $d$   with an infinite residue field  and $I$  an $\mathfrak m$-primary ideal. Is it always the case that $r(I) \leq \n(I)+d$?
\end{question}

In this paper, we  answer the above question raised by Marley for two-dimensional Cohen-Macaulay local rings.  We observe that the stability index of the Ratliff-Rush filtration with respect to $I$, denoted by $\rho(I)$, plays an important role in understanding the relation between $r(I)$ and $\n(I).$ Recall that the \textit{Ratliff-Rush filtration} of $I$, is the filtration $\mathcal{F}=\{\widetilde{I^{n}}\}_{n \geq 0},$ such that $\widetilde{I^{n}}=I^{n}$ for sufficiently large $n.$ The \textit{stability index} of the Ratliff-Rush filtration with respect to $I$ is defined as:
\begin{center}
    $\rho(I):=\min\{i \geq 1:\widetilde{I^{n}}=I^{n}$ for all $n \geq i \}$.
\end{center}

We prove the following result.

\begin{theorem}
    Let $(R, \mathfrak m)$ be a two dimensional Cohen-Macaulay local ring, and $I$ an $\mathfrak m$-primary ideal.  If $\n(I) \neq \rho(I)-1 $, then $r(I) \geq \n(I)+2.$ But if  $\n(I) = \rho(I)-1 $, then $r(I) \leq \n(I)+2$.
\end{theorem}

It is worth examining whether we can extend Theorem 1.2 for higher-dimensional Cohen-Macaulay local rings. The main difficulty in generalizing this result is that there is no simple relation between $\rho(I)$ and $\rho\left(I/(x)\right),$ where $x $ is a superficial element for $I$.  Under some assumptions on the depth of the associated graded ring, we prove the following result.

\begin{proposition} 
    Let $(R, \mathfrak{m})$ be a   Cohen-Macaulay local ring of dimension $d \geq 3$ and $I$ an $\mathfrak{m}$-primary ideal such that $I$ is integrally closed in $R$ and $\depth \gr(I)= d-2.$ If $\n(I)=-(d-3),$ then $r(I)  \geq \n(I)+d.$ 
\end{proposition}

Marley raised another intriguing question  \cite[Question 2.2]{tjm} regarding the existence of a minimal reduction $J$ of $I$. He proposed that:

\begin{question}
    Let $(R,\mathfrak m)$ be a Cohen-Macaulay local ring of dimension $d$ with an infinite residue field and $I$ an $\mathfrak m$-primary ideal. Does there always exist a minimal reduction $J$ of $I$ such that $r_{J}(I)=\n(I)+d?$
\end{question}

We provide an example (see Example \ref{komal}) in a two-dimensional Cohen-Macaulay local ring where the above situation does not occur. This gives a negative answer to the above question.

If $\underline{\mathbf{x}}=x_{1}, x_{2},\ldots,x_{d-1}$ is a superficial sequence for $I$ (see Definition \ref{cute}), we aim to comprehend the relation between $\rho(I),$ $\rho(I/(x_{1})),$ $\ldots,$ $\rho(I/(\underline{\mathbf{x}}))$.
 If the Ratliff-Rush filtration with respect to $I$ behaves well mod a superficial sequence $\underline{\mathbf{x}}$ (see Definition \ref{justin}), then $\rho(I) \geq \rho(I/(x_{1})) \geq \ldots \geq \rho(I/(\underline{\mathbf{x}}))$ (see Lemma \ref{fd}). This provides a necessary condition for the Ratliff-Rush filtration with respect to $I$ to behave well mod a superficial sequence. 
We further observe that, if $x$ is a superficial element for $I,$ then for sufficiently large $n,$  $\frac{I^{n}+(x)}{(x)}=\frac{\widetilde{I^{n}}+(x)}{(x)}=\widetilde{\frac{I^{n}+(x)}{(x)}}$. This fact  is closely related to the surjectivity of the map  $\pi^{n}_{0,1}$ (see Section 4). Set $\omega(I):=\min \big \{k \in \mathbb N : \pi^{n}_{0,1}$ is surjective for all $n \geq k \big\}.$ We call this number the \textit{surjectivity index} of the Ratliff-Rush filtration with respect to $I.$ Note that if $\omega(I)=1,$ then the Ratliff-Rush filtration with respect to $I$ behaves well mod a superficial element. It is of some interest to find an upper bound for $\omega(I)$ (see Proposition \ref{ta}).

Northcott in \cite{no}  investigated the relation between the Hilbert-Samuel polynomial and the Hilbert-Samuel function for one-dimensional Cohen-Macaulay rings  and  proved that  $P_{I}(n+1)-H_{I}(n+1) \geq P_{I}(n)-H_{I}(n)$ for all $n.$
However, there are examples that illustrate that when the dimension of the ring is greater than one, the values of $P_{I}(n)-H_{I}(n)$ do not easily correlate with one another. Marley extended this result for higher dimensions in \cite[Theorem 2.4]{tjm}, and proved that if  $(R, \mathfrak m)$ is a Cohen-Macaulay local ring of dimension $d$ and $I$ an $\mathfrak m$-primary ideal such that $\depth \gr(I) \geq d-1,$ then for $0 \leq i \leq d$ and for all $n \in \mathbb{Z}$,
     $(-1)^{d-i} \Delta^{i}(P_{I}(n)-H_{I}(n)) \geq 0.$

We  observe that the assumption on the depth of the associated graded ring can be relaxed to a weaker condition on the good behavior of the Ratliff-Rush filtration with respect to $I$ mod a superficial sequence.  We prove the following result.

\begin{theorem} 
      Let $(R, \mathfrak m)$ be a Cohen-Macaulay local ring of dimension $d \geq 2,$ $I$ an $\mathfrak m$-primary ideal. Suppose that the Ratliff-Rush filtration with respect to $I$ behaves well mod a superficial sequence $x_{1},x_{2},\ldots,x_{d-2} \in I.$ Then for all $0 \leq i \leq d$ and for all  $n \geq \rho(I)$
\begin{center}
          $(-1)^{d-i} \Delta^{i}(P_{I}(n)-H_{I}(n)) \geq 0.$
      \end{center}
 \end{theorem}

As a corollary, we partially answer (see  Corollary \ref{fg}) the following question proposed by Marley \cite[Question 2.3]{tjm} in the affirmative.

  \begin{question} \label{dc}
      Let $(R, \mathfrak m)$ be a Cohen-Macaulay local ring and $I$ an $\mathfrak m$-primary ideal. Suppose $P_{I}(k)=H_{I}(k)$ for some $k \geq 1$. Then is necessarily $\n(I) <k?$
  \end{question}

This paper is organized into four sections. Definitions, introductory ideas, and notations are all covered in Section 2.    In Section 3, we  establish the relation between the postulation numbers of $I$ and  $I/(x)$, where $x $ is a superficial element for $I.$ We further prove Theorem 1.2 and its consequences related to the independence of reduction number. In Section 4,  we prove Theorem 1.5 and  partially answer Marley's Question 1.6 for  $d$-dimensional Cohen-Macaulay local rings. The techniques used in Section 4  are inspired by Puthenpurakal \cite{tjp,tjp2}. We provide examples in support of our claims.

  \section{PRELIMINARIES}

In this section, we recall the basic definitions, and preliminaries, we also introduce some notations which will be used throughout this paper. Throughout, $(R, \mathfrak m)$ will be a Cohen-Macaulay local ring with an infinite residue field, and $I$ an $\mathfrak m$-primary ideal of $R$. The Krull dimension of the ring $R$, denoted by $d$,  will always assumed to be positive.

\begin{definition}
    \normalfont Consider the increasing chain of ideals in $R$
\begin{center}
    $I\subseteq(I^{2}:I)\subseteq(I^{3}:I^{2})\subseteq\ldots\subseteq(I^{n+1}:I^{n})\subseteq\ldots$
\end{center}
Since, $R$ is Noetherian, this chain stabilizes to an ideal $\widetilde{I}=\displaystyle \bigcup_{n \geq 1}(I^{n+1}:I^{n}).$ This ideal $\widetilde{I}$ is known as the \textit{Ratliff-Rush closure} of $I$ or \textit{Ratliff-Rush ideal} associated with $I.$ An ideal $I$ is called \textit{Ratliff-Rush closed} if $\widetilde{I}=I.$ If $\grade I>0$, then $\widetilde{I}$ is the largest unique ideal containing $I,$ which has the same Hilbert polynomial as that of $I.$

\end{definition}

\begin{definition} \label{cute}
    \normalfont An element $x\in I \backslash I^2$ is called a \textit{superficial element} for $I,$ if there exist a positive integer $c$ such that $(I^{n+1}:x) \cap I^{c}=I^{n}$ for all $n \geq c.$ A sequence $x_{1},x_{2},\ldots,x_{j}$ is  a \textit{superficial sequence} for $I,$ if $x_{1}$ is a superficial element for $I$, and $x_{i}$ is a superficial element for $I/(x_{1},x_{2},\ldots,x_{i-1})$ for all $i=2,3,\ldots,j.$ 
    
    If $x $ is a superficial element for $I$, then $(I^{n+1}:x)=I^{n}$ for sufficiently large $n.$    Let $x$ be any superficial element for $I,$  the \textit{stability index} for $x$ is defined as:
    \begin{center}
         $\rho_{x}(I):=\min\{\text{$k \in \mathbb{N}: (I^{n+1}:x)=I^{n}$ 
  for all $n \geq k$}\}.$
     \end{center}
     In \cite[Corollary 2.7]{tjp}, Puthenpurakal  proved that $\rho(I)=\rho_{x}(I)$, and thus $\rho_{x}(I)$ is independent of the choice of the superficial element. 
\end{definition}

The following lemma gives a relation between the Hilbert-Samuel functions and the Hilbert-Samuel polynomials of $I$ and $I/(x),$ where $x$ is a superficial element for $I$. We remark that the proof is similar to the proof \cite[Lemma 2.3]{yw}, but we include the steps for completeness.

 \begin{lemma} \cite[Lemma 2.3]{yw} \label{pushpa}
     Let $x \in I$ be a superficial element for $I$ and $I'=I/(x)$.  Then for all $n \in \mathbb{Z},$
     \begin{enumerate}[(i)]
         \item $H_{I'}(n) \geq  H_{I}(n)-H_{I}(n-1)$ and
         \item $ P_{I'}(n) =  P_{I}(n)-P_{I}(n-1)$
  
     \end{enumerate}
 \end{lemma}

 \begin{proof} Set $S=R/(x)$ and $I'=I/(x).$
 Let us consider the following exact sequence: \begin{equation} \label{1}
     0 \longrightarrow (I^{n+1} : x)/I^{n+1} \longrightarrow R/I^{n+1} \xlongrightarrow{x} R/I^{n+1} \longrightarrow R/(I^{n+1}+(x)) \longrightarrow 0. 
  \end{equation}

From the exact sequence (\ref{1}), we have $\lambda ((I^{n+1}:x)/I^{n+1})=\lambda (R/(I^{n+1}+(x))).$
Now
\begin{equation} \label{2}
    \begin{split}
    H_{I'}(n+1) & = \lambda (S/I'^{n+1})\\
 & = \lambda (R/(I^{n+1}+(x))) \\
 & = \lambda ((I^{n+1}:x)/I^{n+1})\\
 & = \lambda ((I^{n+1}:x)/I^{n})+ \lambda (I^{n}/I^{n+1})\\
 & \geq  \lambda (I^{n}/I^{n+1})\\
 & = \lambda (R/I^{n+1})-\lambda (R/I^{n})\\
 & = \text{$H_{I}(n+1)-H_{I}(n)$ for all $n$}  
\end{split}
\end{equation}

 Therefore, 
 \text{$H_{I'}(n) \geq  H_{I}(n)-H_{I}(n-1)$
 for all $n$.}  From Definition \ref{cute}, we have $\lambda((I^{n+1}:x)/I^{n})=0$ for sufficiently large $n.$
Thus, part $(ii)$ follows immediately from $(i)$.
\end{proof}

For future references, we recall the following theorem:
\begin{theorem} \cite[Theorem 4.3]{lth2} \label{hoa} Let $(R, \mathfrak m)$ be a two dimensional Cohen-Macaulay local ring and $I$ an $\mathfrak m$-primary ideal. Suppose that $P_{I}(n_{0})=H_{I}(n_{0})$ for some $n_{0} \geq e_{2}(I)+1.$ Then $n(I) < n_{0}.$
\end{theorem}

\begin{notation}
\normalfont
\textbf{I.} For a graded module $M=
\displaystyle \bigoplus_{n \geq 0}M_{n}$,  over a graded ring $A=
\displaystyle \bigoplus_{n \geq 0}A_{n}$, $M_{n}$ denotes its $n$-th graded piece and $M(-j)$ denotes the graded module $M$ shifted  by degree $j$.

\textbf{II.} The \textit{Rees Ring} of $I$ is  $\mathcal{R}(I)=\displaystyle \bigoplus_{n \geq 0}I^{n}t^{n} \subseteq R[t]$, where $t$ is an indeterminate over $R$. Set $\mathcal{R}(I)_{+}=\displaystyle \bigoplus_{n \geq 1}I^{n}t^{n}$ and $\mathfrak M=\mathfrak M_{\mathcal{R}(I)}= \mathfrak m \bigoplus \mathcal{R}(I)_{+}.$

\textbf{III.}   
\end{notation}

\begin{definition}
    Let $M$ be a graded $\mathcal{R}(I)$-module. We define 
    \begin{center}
        $\ex(M)=\sup\{n \in \mathbb{Z}: M_{n} \neq 0 \}$.
    \end{center}

If $\mathfrak{I}$ is a homogeneous ideal in $\mathcal{R}(I),$ then  $\mathcal{H}^{i}_{\mathfrak{I}}(M)$ denotes the $i$-th local cohomology module of $M$ with support in $\mathfrak{I}.$  For the reference of local cohomology, we use \cite{bs}. If $M$ is a finitely generated $\mathcal{R}(I)$-module, then for each $i\geq 0$, $\left(\mathcal{H}^{i}_{\mathcal{R}(I)_{+}}(M)\right)_{n}=0$ for sufficiently large $n.$ Define $a_{i}(M)=\ex\left(\mathcal{H}^{i}_{\mathcal{R}(I)_{+}}(M)\right)$. The \textit{Castelnuovo-Mumford regularity} of $M$ is defined as
     \begin{center}
         $\reg(M)=\max\{a_{i}(M)+i: 0 \leq i \leq \dim M\}.$
     \end{center}
\end{definition}

\begin{definition}
Let $x$ be superficial for $I$,  set $S=R/(x)$ and $I'=I/(x).$ We recall the second fundamental exact sequence from   literature  \cite[6.2, 6.3]{tjp}. For each $n \geq 1,$ we have the following exact sequence of $R$-modules:
\[0\longrightarrow \frac{(I^{n+1}:x)}{I^{n}} \longrightarrow \frac{R}{I^{n}} \xlongrightarrow{\psi_{n}^{x}} \frac{R}{I^{n+1}} \longrightarrow \frac{S}{I'^{n+1}}\longrightarrow 0 ,\]
where $\psi_{n}^{x}(a+I^{n})=xa+I^{n+1}$ for all  $a \in R.$ This sequence induces an exact sequence of $\mathcal{R}(I)$-modules, called the second fundamental exact sequence:
\begin{equation}
\label{eq:u}
    0\longrightarrow \mathcal{B}(x,R) \longrightarrow L^{I}(R)(-1
) \xlongrightarrow{\psi_{x}} L^{I}(R) \longrightarrow L^{I'}(S) \longrightarrow 0,
\end{equation}
 where $\psi_{x}$ is multiplication by $x,$ and $\mathcal{B}(x,R)= \displaystyle \bigoplus_{n \geq 0} \frac{(I^{n+1}:x)}{I^{n}}.$
 Note that, $(I^{n+1}:x)=I^{n}$ for sufficiently large $n.$ Thus, $\mathcal{B}(x,R)$ has finite length. So, $\mathcal{H}^{0}_{\mathfrak M}\left(\mathcal{B}(x,R)\right)$ $=\mathcal{B}(x,R).$ The sequence in (\ref{eq:u}) induces a long exact sequence of local cohomology modules. For convenience from now onwards, we will write $\mathcal{H}^{i}_{\mathfrak M}\left(L^{I}(R)\right)=\mathcal{H}^{i}\left(L^{I}(R)\right)$ for all $i.$
\begin{equation} \label{4}
    \begin{aligned}
        0 \longrightarrow \mathcal{B}(x,R)  &  \longrightarrow \mathcal{H}^{0}\left(L^{I}(R)\right)(-1) \longrightarrow \mathcal{H}^{0}\left(L^{I}(R)\right) \longrightarrow \mathcal{H}^{0}\left(L^{I'}(S)\right) \\
     & \longrightarrow \mathcal{H}^{1}\left(L^{I}(R)\right)(-1) \longrightarrow \mathcal{H}^{1}\left(L^{I}(R)\right) \longrightarrow \mathcal{H}^{1}\left(L^{I'}(S)\right) \ldots
    \end{aligned}
\end{equation}
\end{definition}

\begin{definition}
    A graded $\mathcal{R}(I)$-module $M$ is called $*$-$Artinian,$ if every descending chain of graded submodules of $M$ terminates. From \cite[Proposition 4.4, 5.4]{tjp}, $\mathcal{H}^{i}(L^{I}(R))$ are $*$-Artinian for every $i=0,1,2,\ldots d-1.$  Define $b^{I}_{i}(R)=\ex\left(\mathcal{H}^{i}(L^{I}(R))\right)$ for all $i=0,1,2,\ldots d-1.$ From \cite[5.4]{tjp}, $b^{I}_{i}(R) \leq a_{i+1}(\gr(I))-1.$ Set $b^{I}=\max\{b^{I}_{i}(R): i= 0,1,\ldots,d-1\}.$ By \cite[Proposition 4.7]{tjp}, $\mathcal{H}^{0}\left(L^{I}(R)\right)= \displaystyle \bigoplus_{n \geq 0} \frac{\widetilde{I^{n+1}}}{I^{n+1}}.$
\end{definition}

\begin{observation} \normalfont

\textbf{I.} Note that the zeroth local  cohomology $\mathcal{H}^{0}\left(L^{I}(R)\right)$ of $L^{I}(R)$   does not have negatively graded components i.e., $b_{0}^{I}(R) \geq 0$.  From the long exact sequence of local cohomology modules (\ref{4}), we have
\begin{center}
     $ \longrightarrow \left(\mathcal{H}^{0}\left(L^{I}(R)\right)\right)_{n} \longrightarrow \left(\mathcal{H}^{0}\left(L^{I'}(S)\right)\right)_{n} \longrightarrow \left(\mathcal{H}^{1}\left(L^{I}(R)\right)\right)_{n-1} \longrightarrow \left(\mathcal{H}^{1}\left(L^{I}(R)\right)\right)_{n} $
\end{center}
When $b_{1}^{I}(R) < -1$, then at $n=b_{1}^{I}(R)+1,$ we have 
\begin{center}
     $0 \longrightarrow \left(\mathcal{H}^{1}\left(L^{I}(R)\right)\right)_{b_{1}^{I}(R)} \longrightarrow \left(\mathcal{H} ^{1}\left(L^{I}(R)\right)\right)_{b_{1}^{I}(R)+1} $
\end{center}
         which implies that $\left(\mathcal{H}^{1}\left(L^{I}(R)\right)\right)_{b_{1}^{I}(R)}=0$, a contradiction. Therefore, $b_{1}^{I}(R) \geq -1. $
         
         \textbf{II.} If $I$ is an integrally closed ideal in $R$, then from \cite[Page 648]{si}, $I'$ is an integrally closed ideal in $S$. Thus, $\widetilde{I'}=I'$ which implies $\left(\mathcal{H}^{0}\left(L^{I'}(S)\right)\right)_{0}=0.$ Therefore, $b_{1}^{I}(R) \geq 0.$
         \end{observation}

\section{Relation between the reduction number and the postulation number of an $\mathfrak m$-primary ideal}

 Marley \cite[Lemma 2.8]{tjm} proved that if $(R,\mathfrak m)$ is a Cohen-Macaulay local ring of dimension $d$ with an infinite residue field, and $I$ an $\mathfrak m$-primary ideal of $R$ such that $x$ is a superficial element for $I$ and a  $ \gr(I)$-regular element, then $\n(I)+1=\n(I'),$ where $I'=I/(x)$. Further, in  \cite[Theorem 2.15]{tjm}, he proved that if $\depth \gr(I) \geq d-1,$ then $r(I)=\n(I)+d$. Inspired by Marley's work, one can ask whether these results can be  extended  without any  assumption on the depth of the associated graded ring. In this section,  we attempt to extend Marley's results. 

In the next proposition, we observe that  $\rho(I)$  plays an important role in describing the behaviour of the postulation number of $I$ and the postulation number of $I$ going mod a superficial element.

 \begin{proposition} \label{riya}
        Let $(R, \mathfrak m)$ be a  Cohen-Macaulay local ring of dimension $d$ and $I$ an $\mathfrak m$-primary ideal. Let $x \in I$ be a superficial element for $I$ and $I'=I/(x)$. 
        \begin{enumerate}[(i)]
        \item If $\n(I) > \rho(I)-1 $, then $\n(I') = \n(I)+1.$
          \item If $\n(I)=\rho(I)-1 $, then $\n(I') \leq \n(I)+1.$ 
            \item  If $\n(I) < \rho(I)-1 $, then $\n(I') > \n(I)+1.$
          
        \end{enumerate}
\end{proposition}

\begin{proof}
Set $S=R/(x)$. Since $x$ is a superficial element for $I$, therefore, from Definition \ref{cute}, we have $(I^{n+1}:x)=I^{n}$  for all $n \geq \rho(I)$. By Lemma \ref{pushpa}, we have  $H_{I'}(n) =  H_{I}(n)-H_{I}(n-1)$ for all $ n \geq \rho(I)+1.$ By  (\ref{2}),  at $\rho(I)$, we have  $H_{I'}(\rho(I)) > H_{I}(\rho(I))-H_{I}(\rho(I)-1).$ 
\begin{enumerate}[(i)]

    \item Suppose $\n(I) > \rho(I)-1.$  For any $m > \n(I)+1 \geq \rho(I)+1$, we have $ H_{I'}(m) =  H_{I}(m)-H_{I}(m-1) = P_{I}(m)-P_{I}(m-1)= P_{I'}(m)$  so, $\n({I'}) \leq \n(I)+1$. If $\n(I') < \n(I)+1,$ then $ H_{I'}(\n(I)+1)=P_{I'}(\n(I)+1),$ which implies $ H_{I}(\n(I)+1)-H_{I}(\n(I))=P_{I}(\n(I)+1)-P_{I}(\n(I))$. Thus, $H_{I}(\n(I))=P_{I}(\n(I))$,   a contradiction. Hence, $\n(I') = \n(I)+1.$

\item Suppose $\n(I) = \rho(I)-1.$ Then for any $m \geq \rho(I)+1$, we have $ H_{I'}(m)= H_{I}(m)-H_{I}(m-1)=P_{I}(m)-P_{I}(m-1)=P_{I'}(m)$. Therefore, we get $\n(I') < \rho(I)+1$ which implies $\n(I') \leq \n(I)+1.$

    \item Suppose $\n(I)<\rho(I)-1$.  We have  $ H_{I'}(\rho(I)) > P_{I'}(\rho(I))$ as $H_{I'}(\rho(I))>H_{I}(\rho(I))-H_{I}(\rho(I)-1) = P_{I}(\rho(I))-P_{I}(\rho(I)-1) = P_{I'}(\rho(I)).$ Therefore, we get $\n(I') \geq \rho(I)>\n(I)+1$ which implies $ \n(I') > \n(I)+1. $ \qedhere
    \end{enumerate}  \end{proof}

\begin{remark} \label{singh}
\normalfont
From Proposition \ref{riya} $(ii)$, if $\n(I)=\rho(I)-1$, then for $m \geq \rho(I)+1$, we have $H_{I'}(m) = P_{I'}(m).$ Thus, $\n(I') \leq \rho(I)= \n(I)+1.$ At $\rho(I)-1,$  $ H_{I}(\rho(I)-1) \neq P_{I}(\rho(I)-1) $ which implies $ H_{I}(\rho(I)-1) = P_{I}(\rho(I)-1)+l$ for some non-zero integer $l.$ If $l<0$, then  $H_{I}(\rho(I)-1)<P_{I}(\rho(I)-1).$
From Lemma \ref{pushpa}, at $\rho(I)$, we have  $H_{I'}(\rho(I))=H_{I}(\rho(I))-H_{I}(\rho(I)-1)>P_{I}(\rho(I))-P_{I}(\rho(I)-1)= P_{I'}(\rho(I)).$ Therefore, $H_{I'}(\rho(I))> P_{I'}(\rho(I)),$ which implies $\n(I') \geq \rho(I)=\n(I)+1.$   Thus, if $H_{I}(n) < P_{I}(n)$ for $n = \rho(I)-1$, we have $ \n(I') = \n(I)+1$.
\end{remark}

The following examples illustrate the above proposition.

\begin{example} \label{kkr}
\normalfont
\cite[Example 5.1]{jds} Let $R=\mathbb{Q}\llbracket x,y\rrbracket$ and $I=(x^6,x^4y,xy^5,y^6)$.  The Hilbert polynomial for $I$ is 
\begin{center}
    $P_{I}(n)=30\dbinom{n+1}{2}-10n+3.$
\end{center}
Note that $P_{I}(4) = H_{I}(4),$ thus from Theorem \ref{hoa}, $\n(I) <4.$ Using Macaulay 2, we get $P_{I}(0) \neq H_{I}(0)$ and $P_{I}(n) = H_{I}(n)$ for $1\leq n\leq3,$ therefore,  $\n(I)=0.$
 From \cite[Proposition 1.2(2)]{rv}, $p=x^4y+y^6$ is a superficial element for $I$ as $e_{0}(I)=30=e_{0}(I/(p)), e_{1}(I)=10=e_{1}(I/(p)).$ Since $x^3y^4 \in (I^2:I)$ but $x^3y^4 \notin I$, therefore $\widetilde{I} \neq I,$ which implies $\rho(I) \geq 2$ from \cite[Remark 1.6]{rs}. Here, $\n(I) \neq \rho(I)-1.$ Let $S=R/(p)$ and $I'=I/(p).$ The Hilbert polynomial for $I'$ is $P_{I'}(n)=30n-10.$ 
 Here,   $P_{I'}(2) \neq H_{I'}(2)$,  thus $\n(I')\geq 2 > \n(I)+1.$
\end{example}

\begin{example} \cite[Example 2.7]{tjm} \label{richa}
          \normalfont Let  $R=\mathbb{Q}\llbracket x,y\rrbracket$ and $I=(x^7,x^6y,xy^6,y^7)$.  The Hilbert polynomial for $I$ is 
\begin{center}
    $P_{I}(n)=49\dbinom{n+1}{2}-21n.$
\end{center}
Note that $P_{I}(5) = H_{I}(5),$ thus from Theorem \ref{hoa}, $\n(I) <5.$ Using Macaulay 2, we get  $P_{I}(4) \neq H_{I}(4),$ therefore, $\n(I)=4.$
  From \cite[Proposition 1.2(2)]{rv},  $p=x^7+y^7$ is a superficial element for $I$ as $e_{0}(I)=49=e_{0}(I/(p))$ and $e_{1}(I)=21=e_{1}(I/(p))$. Here $J=(x^{7},y^{7}
  )$ is a minimal reduction of $I$ with $r_{J}(I)=5$ and $(I^{6}:x)= I^{5}$, therefore, from \cite[Proposition 4.3]{hjls}, $\rho(I) \leq 5.$  Using Macaulay 2, we get $\rho(I)=5,$ as $(I^{5}:x)\neq I^{4},$ which implies  $\n(I)=\rho(I)-1.$
            Let $S=R/(p)$, $I'=I/(p)$ and $J'=J/(p).$ The Hilbert polynomial for $I'$ is
                   $P_{I'}(n)=49n-21.$ Note that, $r_{J}(I) \geq r_{J'}(I')=\n(I')+1$  by \cite[Proposition 4.10]{oo}, which implies $\n(I')\leq 4 < \n(I)+1.$ 
        
\end{example}

We now answer Marley's Question \ref{csk} for two-dimensional Cohen-Macaulay local rings. We observe that both  inequalities are possible under certain assumptions. 
 
\begin{theorem} \label{abhishek}

 Let $(R, \mathfrak m)$ be a two dimensional  Cohen-Macaulay local ring, and $I$  an $\mathfrak m$-primary ideal. If $\n(I) \neq \rho(I)-1, $ then $r(I) \geq \n(I)+2.$
 \end{theorem}

 \begin{proof}
     Suppose $\n(I) \neq \rho(I)-1 $. Let $x \in I$ be a superficial element for $I.$ Set $S=R/(x)$ and $I'=I/(x).$ By Proposition \ref{riya} $(i)$ and $(iii)$, $ \n(I') \geq \n(I)+1$. Here, $S$  a one-dimensional Cohen-Macaulay local ring, hence by \cite[Proposition 4.10]{oo}, we have $\n(I')+1=r(I').$ Therefore,  
 $\n(I)+1 \leq \n(I')=r(I')-1\leq r(I)-1$. Thus, $r(I) \geq \n(I)+2.$ \end{proof}

We recall  Example \ref{kkr} to illustrate the above theorem.

\begin{example} \normalfont
\cite[Example 5.1]{jds} Let $R=\mathbb{Q}\llbracket x,y\rrbracket$ and $I=(x^6,x^4y,xy^5,y^6)$.  Note that
$J=(x^4y,x^6+xy^5+y^6)$ is a minimal reduction for $I$ with  $r_{J}(I)=3 \geq \n(I)+2.$  From \cite[Theorem 3.3']{yw}, if there exist a minimal reduction $K$ of $I$ such that $r_{K}(I)\geq \n(I)+2,$ then $r(I)$ is independent. 
Therefore, $r(I) \geq \n(I)+2.$
\end{example}

\begin{remark}
    Wu \cite[Remark 1.(b)]{yw} proved that,  if $(R,\mathfrak m)$ is a Cohen-Macaulay ring of dimension two, and $I$ is an $\mathfrak m$-primary ideal with $J \subseteq I$ a minimal reduction such that $r_{J}(I) \leq \n(I)+1$ then $\n(I)=\rho(I)-1.$ The next proposition gives us the converse of the above-stated result.
\end{remark}

\begin{theorem} \label{anjali}

 Let $(R, \mathfrak m)$ be a two dimensional  Cohen-Macaulay local ring, and $I$ an $\mathfrak m$-primary ideal. If $\n(I) = \rho(I)-1 $, then $r(I) \leq \n(I)+2.$
 \end{theorem}

 \begin{proof}
 
    Let $J=(x,y)$ be a minimal reduction of $I$ such that $x \in I$ is a superficial element for $I$. We have $r_{J'}(I') \leq r_{J}(I)$ where $J'=J/(x)$ and $I'=I/(x).$ 
       Since $\n(I)= \rho(I)-1$, by Proposition \ref{riya} $(ii)$, $\n(I')\leq \n(I)+1.$ Now,  $S=R/(x)$ is a one-dimensional Cohen-Macaulay ring, by \cite[Proposition 4.10]{oo}, we have $r(I')=\n(I')+1$ which implies $r(I') \leq \rho(I)+1.$  If possible  let $r_{J}(I) > \rho(I)+1.$ By \cite[Lemma 3.1]{ms}, we have $\widetilde{I^{\rho(I)+1}} \neq I^{\rho(I)+1} $ as  $r(I') \leq \rho(I)+1 < r_{J}(I)$, which is a contradiction. Therefore, $r_{J}(I) \leq \rho(I)+1$ which implies  $r_{J}(I) \leq \n(I)+2.$ Also $r(I) \leq r_{J}(I) $ for any minimal reduction $J$ of $I$. Thus, $r(I) \leq \n(I)+2.$ 
\end{proof}

       The following example illustrates the above theorem.

\begin{example} \cite[Example 4.6]{rtt}
    \normalfont Let $R=\mathbb{Q}\llbracket x,y\rrbracket$ and $I=(x^7,x^6y,x^2y^5,y^7).$ Since $x^5y^4 \in (I^2:I)$ but $x^5y^4 \notin I$, therefore, $\widetilde{I}\neq I.$   Thus, $\depth \gr(I)=0.$ The Hilbert polynomial for $I$ is 
\[P_{I}(n)=49\dbinom{n+1}{2}-21n+3.\]

Note that $P_{I}(4) = H_{I}(4),$ thus from Theorem \ref{hoa}, $\n(I) <4.$ Using Macaulay 2, we get $P_{I}(3) \neq H_{I}(3),$ therefore, $\n(I)=3.$  
From \cite[Proposition 1.2(2)]{rv}, $p=x^7+y^7$ is a superficial element for $I$ as $e_{0}(I)=49=e_{0}(I/(p)), e_{1}(I)=21=e_{1}(I/(p)) $. Note that  $J=(x^7,y^7)$ is a minimal reduction of $I$ with $r_{J}(I)=4$ and $(I^{5}:x)=I^{4},$ therefore from \cite[Proposition 4.3]{hjls}, $\rho(I)\leq 4.$ Using Macaulay 2, we get $\rho(I)= 4,$ as $(I^{4}:x)\neq I^{3},$ which is equal to $\n(I)+1.$ Therefore,  $r(I) \leq r_{J}(I)  \leq \n(I)+2.$
\end{example}

As an immediate application of Theorem \ref{abhishek}, we provide an alternative proof for a proposition by Hoa \cite[Proposition 3.7]{lth} regarding the independence of $r(I)$.

\begin{corollary}\label{bb}\cite[Proposition 3.7]{lth}  Let $(R, \mathfrak m)$ be a two dimensional  Cohen-Macaulay local ring, and $I$  an $\mathfrak m$-primary ideal. If $\n(I) \neq \rho(I)-1$, then the reduction number is independent.
 \end{corollary}

\begin{proof}
    Since $\n(I) \neq \rho(I)-1$, by Theorem \ref{abhishek}, we have $r(I) \geq  \n(I)+2.$ Thus, the conclusion follows from \cite[Theorem 3.3]{yw}.
\end{proof}


 In the corollary below, we outline the condition under which $r(I)=\n(I)+2$ without any assumption on the depth of the associated graded ring.

\begin{corollary}
    Let $(R,\mathfrak m)$ be a two dimensional Cohen-Macaulay local ring, and $I$ an $\mathfrak m$-primary ideal. Suppose $\n(I) = \rho(I)-1$ and $H_{I}(\n(I)) < P_{I}(\n(I)),$ then  $r(I)=\n(I)+2.$
\end{corollary}

\begin{proof}
    Let $J=(x,y)$ be a minimal reduction for $I,$ where $x \in I$ is a superficial element for $I.$   Since $\n(I) = \rho(I)-1$ and $H_{I}(\n(I)) < P_{I}(\n(I)),$ by  Remark \ref{singh}, we have $\n(I')=\n(I)+1,$ where $I'=I/(x).$ Also, $S=R/(x)$ is a one-dimensional Cohen-Macaulay local ring, thus by \cite[Proposition 4.10]{oo}, it follows that $r(I')=\n(I')+1=\n(I)+2,$ also $r(I') \leq r_{J}(I),$ which implies $\n(I)+2 \leq r_{J}(I).$ By \cite[Theorem 3.3']{yw}, $r(I)$ is independent. Thus, $\n(I)+2 \leq r(I).$ Again by Theorem \ref{anjali}, we have $r(I) \leq \n(I)+2.$ Therefore, $r(I)=\n(I)+2.$
\end{proof}

Next, we generalize Theorem \ref{abhishek} for  Cohen-Macaulay local rings of dimension $d \geq 3$ with $\depth \gr(I) =d-2$ under certain assumptions on $I$ and $\n(I).$ It is very well known from the literature that the integral closure of an ideal behaves well reducing modulo a superficial element.  Itoh proved that:

\begin{lemma}\cite[Page 648]{si} \label{maa}
If $I$ is an $\mathfrak{m}$-primary ideal which is integrally closed in a Cohen–Macaulay local ring $(R, \mathfrak{m})$ of dimension $d \geq 2$, then (at least after passing to a faithfully flat extension) there exists a superficial element $x \in I$ such that $I/(x)$ is integrally closed in $R/(x)$.    
\end{lemma}

In the next two propositions, we establish the relation between $\n(I)$ and $r(I)$ for Cohen-Macaulay local rings of dimension $d \geq 3$, with $\depth \gr(I)=d-2.$

\begin{proposition} \label{papa}
    Let $(R, \mathfrak{m})$ be a   Cohen-Macaulay local ring of dimension $d \geq 3$ and $I$ an $\mathfrak{m}$-primary ideal such that $I$ is integrally closed in $R$ and $\depth \gr(I)= d-2.$ If $\n(I)=-(d-3),$ then $r(I)  \geq \n(I)+d.$  Moreover, the reduction number is independent.
\end{proposition}

\begin{proof}
   By Lemma \ref{maa}, we choose a superficial sequence  $x_{1}, x_{2},  \ldots, x_{d-2}$ for $I$ such that $I'$ is integrally closed in $S$, where $I'=I/(x_{1}, x_{2}, \ldots, x_{d-2})$ and $S=R/(x_{1}, x_{2}, \ldots, x_{d-2}).$  Since  $\depth \gr(I) =d-2$, therefore, from \cite[Lemma 2.8]{tjm}, $\n(I')=\n(I)+d-2.$   Now $\n(I)=-(d-3)$, implies $\n(I')=1.$ Note that $S$ is a two-dimensional Cohen-Macaulay ring with $\depth \gr(I')=0$. We claim that $\rho(I') \geq 3$. As $\depth \gr(I')=0$, by \cite[1.3]{hjls} $\rho(I') > 1.$ When  $\rho(I') =2$, then $I' \neq \widetilde{I'}$. But this is a contradiction to the assumption that $I'$ is an integrally closed ideal. 
   Thus, claim holds and $\n(I') \neq \rho(I')-1$. Hence,  by Theorem  \ref{abhishek},  $r(I') \geq \n(I')+2.$ Therefore, by \cite[Lemma 2.14]{tjm}, $r(I) = r(I')  \geq \n(I')+2 = \n(I)+d.$ Further, the reduction number is independent follows from \cite[Theorem 4.3]{yw}.
\end{proof}

        \begin{remark}
            In the hypothesis of Proposition \ref{papa}, $I$ being integrally closed in $R$ is a necessary condition. Let $R=\mathbb{Q}\llbracket x,y,z\rrbracket$ and $I=(x^{3},y^{3},z^{3},xy^{2},x^{2}z,xz^{2},y^{2}z,yz^{2},xyz)$. Note that $I$ is not integrally closed as $x^{2}y \in \Bar{I}$ but $x^{2}y \notin I.$ By \cite[Proposition 1.12]{rs},  $\depth \gr(I)=1$. The Hilbert polynomial for $I$ is 
\begin{center}
     $P_{I}(n)=27\dbinom{n+2}{3}-18\dbinom{n+1}{2}+n+1.$
\end{center}
              Here $P_{I}(0) \neq H_{I}(0)$ and $P_{I}(n) = H_{I}(n)$ for all $n \geq 1,$ so $\n(I) = 0.$ Using Macaulay 2, $K=(\frac{5}{3}x^{3}+6xy^{2}+\frac{7}{6}y^{3}+\frac{1}{10}x^{2}z+xyz+\frac{3}{10}y^{2}z+\frac{7}{4}xz^{2}+\frac{1}{4}yz^{2}+z^{3},\frac{6}{5}x^3+\frac{1}{4}xy^{2}+\frac{3}{5}y^{3}+\frac{1}{4}x^2z+\frac{3}{8}xyz+\frac{6}{5}y^{2}z+\frac{3}{10}xz^{2}+\frac{7}{9}yz^{2}+z^{3},\frac{1}{6}x^{3}+\frac{2}{3}xy^{2}+5y^{3}+\frac{1}{5}x^{2}z+\frac{7}{8}xyz+\frac{3}{4}y^{2}z+xz^{2}+\frac{5}{3}yz^{2},3z^{3})$ is a minimal reduction of $I$ and $2=r_{K}(I) < \n(I)+3$ also $r_{K}(I) \geq r(I),$ thus $r(I) < \n(I)+3.$         
        \end{remark}

        \begin{proposition} \label{purvi}
             Let $(R, \mathfrak{m})$ be a   Cohen-Macaulay local ring of dimension $d \geq 3$ and $I$  an $\mathfrak{m}$-primary ideal such that $\depth \gr(I)= d-2.$ If $\n(I) < -(d-3),$ then $r(I)  \geq \n(I)+d.$ Further, in this case the reduction number is independent.
        \end{proposition}

        \begin{proof}
Let  $x_{1}, x_{2},  \ldots, x_{d-2}$ be a superficial sequence for $I$ such that  $I'=I/(x_{1}, \ldots, x_{d-2})$ and $S=R/(x_{1}, \ldots, x_{d-2}).$  Since the $\depth \gr(I)=d-2$, therefore, from \cite[Lemma 2.8]{tjm}, $\n(I')=\n(I)+d-2.$ It is given that $\n(I) < -(d-3)$ which implies $\n(I')<1.$ Since $\depth \gr(I')=0$,  by \cite[1.3]{hjls} we have $\rho (I')>1.$ Which implies $\n(I') \neq \rho (I')-1.$
Hence, by Theorem  \ref{abhishek}, we have $r(I') \geq \n(I')+2.$ Therefore, by \cite[Lemma 2.14]{tjm}, $r(I) = r(I') \geq \n(I')+2 = \n(I)+d.$
             \end{proof}

We conclude this section by providing a counter-example to the question raised by Marley \cite[Question 2.2]{tjm}.  

\begin{example} \label{komal}
     \normalfont Let $R=\mathbb{Q}\llbracket x,y\rrbracket$ and $I=(x^8,x^6y,xy^7,y^8)$. Since $x^5y^6 \in (I^2:I)$ but $x^5y^6 \notin I$, therefore, $\widetilde{I}\neq I.$   Thus, $\depth \gr(I)=0.$ The Hilbert polynomial for $I$ is 
\[P_{I}(n)=56\dbinom{n+1}{2}-21n+13.\]
Note that $P_{I}(14) = H_{I}(14),$ thus from Theorem \ref{hoa}, $\n(I) <14.$ Using Macaulay 2, we get  $P_{I}(1) \neq H_{I}(1)$,  and $P_{I}(n) = H_{I}(n)$ for all $2 \leq n \leq 13$, thus $\n(I)=1.$
 Here $J=(x^6y,x^8+xy^7+y^8)$ is a minimal reduction of $I$ as $JI^5=I^6$ and $\lambda(R/J)=56=e_{0}(I)$ 
with  $r_{J}(I)=5 > \n(I)+2.$ By \cite[Theorem 3.3']{yw}, we have $r_{J}(I)=5$ for every minimal reduction $J$ of $I.$ Therefore, $r_{J}(I) > \n(I)+2$ for every minimal reduction $J$ of $I.$ 
\end{example}

\section{Ratliff-Rush filtration mod a superficial sequence}

We suggest the reader that while reading this section, it is a good idea to refer to the paper \cite{tjp2}. In this section, we  extend the theorem by Marley \cite[Theorem 2.3]{tjm} for Cohen-Macaulay local rings of dimension $d$. We observe that the stronger condition on the depth of the associated graded ring can be relaxed to a weaker condition on the good behaviour of the Ratliff-Rush filtration with respect to $I$.   We recall the good behaviour of the Ratliff-Rush filtration of an ideal $I$ mod a superficial sequence from \cite[Section 4]{tjp2}.

For a positive integer $s \leq d-1$, let $x_{1}, x_{2},\ldots, x_{s} \in I$ be a superficial sequence for $I$. Set $R_{i}=R/(x_{1}, x_{2},\ldots, x_{i})$ for all $i=1,2,\ldots,s.$ We have the natural map of quotients $\pi_{i,j}:R_{i}\longrightarrow R_{j}$, for every $0\leq i < j\leq s.$ Note that 
$\pi_{i,j}\left(\widetilde{I^{n}R_{i}} \right) \subseteq \widetilde{I^{n}R_{j}} 
$  for all $n \geq 1.$

 This map induces the following map:
$\pi_{i,j}^{n}:\frac{\left(\widetilde{I^{n}R_{i}}\right)}{\left(I^{n}R_{i}\right)} \longrightarrow \frac{\left(\widetilde{I^{n}R_{j}}\right)}{\left(I^{n}R_{j}\right)}$ for every $n \geq 1.$ For a fixed $n \geq 1,$ we observe that $\pi_{i,j}\left(\widetilde{I^{n}R_{i}} \right) =\widetilde{I^{n}R_{j}} $   if and only if $\pi_{i,j}^{n}$ is surjective. For $0 \leq i<j<k \leq s$, we have the following commutative diagram:
\begin{center}
    $\begin{tikzcd}
	{\widetilde{I^{n}R_{i}}/I^{n}R_{i}} \\
	\\
	{\widetilde{I^{n}R_{j}}/I^{n}R_{j}} && {\widetilde{I^{n}R_{k}}/I^{n}R_{k}}
	\arrow["{\pi_{i,j}^{n}}"', from=1-1, to=3-1]
	\arrow["{\pi_{i,k}^{n}}", from=1-1, to=3-3]
	\arrow["{\pi_{j,k}^{n}}"', from=3-1, to=3-3]
\end{tikzcd}$
\end{center}

\begin{remark} \label{xx}
    If $\pi_{i,j}^{n}$ and $\pi_{j,k}^{n}$ are surjective, then $\pi_{i,k}^{n}$ is also surjective.
\end{remark}

\begin{definition} \cite[Definition 4.4]{tjp2}
   The Ratliff-Rush filtration of $I$ behaves well mod superficial sequence $x_{1}, x_{2},\ldots, x_{s} \in I ,$ if $\pi_{0,s}(\widetilde{I^{n}})=\widetilde{I^{n}R_{s}}$ for all $n \geq 1.$
\end{definition}

\begin{definition} \label{justin}
 Let $x_{1}, x_{2},\ldots, x_{s}$ be a superficial sequence for $I$,
then for each $i=1,2,\ldots,s$, we define
\begin{center}
    $\rho\left(I/(x_{1}, x_{2}, \ldots, x_{i})\right):=\min\{k \in \mathbb{N} :I^{n}R_{i}=\widetilde{I^{n}R_{i}}$ for all $ n \geq k \}.$
\end{center}
\end{definition}

\begin{proposition} \label{papaji}
       Let $(R, \mathfrak m)$ be a Cohen-Macaulay local ring of dimension $d \geq 2,$ $I$ an $\mathfrak m$-primary ideal and $x_{1}, x_{2},\ldots, x_{s} \in I $ be a superficial sequence for $I.$ Then $\pi_{0,s}^{n}$ is surjective for all $n \geq \max \{1, b^{I}+2\}.$
\end{proposition}

\begin{proof} Set $R_{i}=R/(x_{1}, x_{2},\ldots, x_{i})$ for all $i=1,2,\ldots,s$ and $\mathfrak M= \mathfrak M_{\mathcal{R}(I)}$.
We will prove by induction on $s.$ Suppose $s=1,$   using the short exact sequence of $\mathcal{R}(I)$-modules \cite[5]{tjp}
   
   \begin{center}
       $0\longrightarrow \gr(I) \longrightarrow L^{I}(R) \longrightarrow  L^{I}(R)(-1) \longrightarrow 0$
   \end{center}
   
   and the corresponding long exact sequence in local cohomolgy modules, we get $b_{i}^{I}(R) \leq a_{i+1}(\gr(I))-1$ for $i=0,1,\ldots, d-1$.
From \cite[2.1]{tjp2} note that, $\pi_{0,1}$ induces a  natural $\mathcal{R}(I)$-linear map $\pi^{*}:L^{I}(R) \longrightarrow L^{I/(x_{1})}(R_{1}).$ This induces a map on the local cohomology modules
   \begin{center}
       $\mathcal{H}^{0}\left(\pi^{*}\right):\mathcal{H}^{0}\left(L^{I}(R)\right) \longrightarrow \mathcal{H}^{0}\left(L^{I/(x_{1})}(R_{1})\right).$ 
\end{center}

       By \cite[proposition 4.7]{tjp} $\mathcal{H}^{0}\left(\pi^{*}\right)_{n}=\pi_{0,1}^{n}$ for all $n \in \mathbb{Z}.$
Note that $\mathcal{H}^{1}\left(L^{I}(R)\right)_{n}=0$ for all  $n \geq b_{1}^{I}(R)+1$. By the second fundamental exact sequence (\ref{eq:u}), we get, 
      \begin{center}
       $\mathcal{H}^{0}\left(L^{I}(R)\right)_{n} \xrightarrow{\pi_{0,1}^{n}} \mathcal{H}^{0}\left(L^{I/(x_{1})}(R_{1})\right)_{n}\longrightarrow 0,$ for all $n \geq b_{1}^{I}(R)+2.$
\end{center}
Thus, $\pi_{0,1}^{n}$ is surjective for all $n \geq b_{1}^{I}(R)+2$ which implies  $\pi_{0,1}^{n}$ is surjective for all $n \geq b^{I}+2.$

 Let $s\geq2,$ we assume that $\pi_{0,s-1}^{n}$ is surjective for all $ n \geq b^{I}+2,$ and prove that $\pi_{0,s}^{n}$ is surjective for all $ n \geq b^{I}+2.$
By $s=1$ case, $\pi_{0,1}^{n}$ is surjective for all $ n \geq b^{I}+2.$ Using the second fundamental exact sequence (\ref{eq:u}) and the corresponding long exact sequence of local cohomolgy modules  (\ref{4}), we get $\mathcal{H}^{i}\left(L^{I/(x_{1})}(R/(x_{1}))\right)_{n}=0$ for all $ n \geq b^{I}+2$ and for all $i=1,2,\ldots,s-1.$ Thus, by induction hypothesis $\pi_{1,s}^{n}$ is surjective for all $ n \geq b^{I}+2.$ Now, using Remark \ref{xx}, and the commutative diagram,  $\pi_{0,s}^{n}$ is surjective for all $ n \geq b^{I}+2.$
\end{proof}

 As an easy consequence of Proposition \ref{papaji}, we have the following corollary.

\begin{corollary}\label{sp}
    Let $(R, \mathfrak m)$ be a Cohen-Macaulay local ring of dimension $d \geq 2,$ $I$ an $\mathfrak m$-primary ideal and $x_{1}, x_{2},\ldots, x_{s} \in I $ be a superficial sequence for $I.$ If $\rho(I) \geq b^{I}+2$, then $\rho(I) \geq \rho\left(I/(x_{1}, \ldots, x_{s})\right).$
\end{corollary}

\begin{proof}
For all $n \geq \rho(I),$ we have $I^{n}=\widetilde{I^{n}}.$ From Proposition \ref{papaji}, we have $\rho(I) \geq b^{I}+2,$ therefore, $\pi_{0,s}^{n}$ is surjective for all $n \geq \rho(I),$ which implies $\pi_{0,s}(\widetilde{I^{n}})=\widetilde{I^{n}R_{s}}$ for all $n \geq \rho(I).$ Thus, $\pi_{0,s}(I^{n})=\widetilde{I^{n}R_{s}}$ for all $n \geq \rho(I),$ which implies $\rho(I) \geq \rho\left(I/(x_{1},x_{2,} \ldots, x_{s})\right).$
    \end{proof}

 We recall the surjectivity index of the Ratliff-Rush filtration with respect to $I$  from the introduction. Note that if the Ratliff-Rush filtration with respect to $I$ mod a superficial element behaves well then $\omega(I)=1.$ Next, we give an upper bound on $\omega(I)$ for two-dimensional Cohen-Macaulay local rings.

 \begin{proposition}  \label{ta}
     Let $(R, \mathfrak m)$ be a Cohen-Macaulay ring of dimension $d \geq 2,$ and $I$ an $\mathfrak m$-primary ideal. If $b^{I}_{1}(R) \neq -1$, then $ \omega(I) = \max\{1, b_{1}^{I}(R)+2\} \leq  \max\{1,a_{2}(\gr(I))+1\}.$
 \end{proposition}

\begin{proof} From Proposition \ref{papaji}, $\pi_{0,1}^{n}$ is surjective for all $n \geq b_{1}^{I}(R)+2$.
  Therefore, $\omega(I) \leq \max \{1,b_{1}^{I}(R)+2\}.$
If possible let, $b_{1}^{I}(R)+2 > \omega(I),$ then $\pi_{0,1}^{n}$ is surjective for all $n \geq b_{1}^{I}(R)+1.$ By the long exact sequence of local cohomology modules (\ref{4}), we get the following exact sequence:

 \begin{center} \begin{equation} \label{5}
        \left(\mathcal{H}^{0}\left(L^{I}(R)\right)\right)_{n} \xrightarrow{\pi_{0,1}^{n}}   \mathcal{H}^{0}\left(L^{I/(x_{1})}(R_{1})\right)_{n} \xrightarrow{\phi_{n}} \left(\mathcal{H}^{1}\left(L^{I}(R)\right)\right)_{n-1} \xrightarrow{\zeta_{n}} \left(\mathcal{H}^{1}\left(L^{I}(R)\right)\right)_{n}.   
 \end{equation} 
   \end{center}
 Note that,   ker$\phi_{n}$$=$ Im$\pi_{0,1}^{n}$ and ker$\zeta_{n}$$=$ Im$\phi_{n}$. Since, $\pi_{0,1}^{n}$ is surjective for all $n \geq b_{1}^{I}(R)+1,$ thus, ker$\phi_{n}$$=$ $\mathcal{H}^{0}\left(L^{I/(x_{1})}(R_{1})\right)_{n}$ for all $n \geq b_{1}^{I}(R)+1.$ Therefore, $\phi_{n}$ is a zero map for all $n \geq b_{1}^{I}(R)+1,$ which implies $\zeta_{n}$ is injective for all $n \geq b_{1}^{I}(R)+1.$
 By substituting $n=b_{1}^{I}(R)+1$ in (\ref{5}), we get $\left(\mathcal{H}^{1}\left(L^{I}(R)\right)\right)_{n}=0,$ which implies that $\left(\mathcal{H}^{1}\left(L^{I}(R)\right)\right)_{b_{1}^{I}(R)}=0,$  a contradiction. Thus,    $b_{1}^{I}(R)+1 < \omega(I),$ which implies $b_{1}^{I}(R)+2 \leq \omega(I).$ 
 Therefore,  $\omega(I)= \max\{1, b_{1}^{I}(R)+2\} \leq \max\{1, a_{2}(\gr(I))+1\}. $
 \end{proof}

 As an immediate corollary we get

\begin{corollary}
    \label{rcb}  Let $(R, \mathfrak m)$ be a Cohen-Macaulay ring of dimension two, $I$ an $\mathfrak m$-primary ideal and $x \in I$ be a superficial element for $I.$ If $ \rho(I) \geq r(I)-1,$ then $\rho(I)  \geq \rho(I/(x)). $
\end{corollary}

\begin{proof}
    Since $ \rho(I) \geq r(I)-1,$ then by \cite[Lemma 1.2]{tmj2}, we have  $a_{2}(\gr(I))+2 \leq r(I),$ which implies  $ \rho(I) \geq \max\{1,a_{2}(\gr(I))\} \geq \omega(I) \leq b_{1}^{I}(R)+2,$ by Proposition  \ref{ta}. Therefore, by Corollary \ref{sp} it follows that $\rho(I)  \geq \rho(I/(x)). $
\end{proof}


In the next corollary, we give a class of $\mathfrak m$-primary ideals for two-dimensional Cohen-Macaulay local rings for which the Ratliff-Rush filtration of $I$  behaves well mod a superficial element. 

\begin{corollary}
 Let $(R, \mathfrak m)$ be a Cohen-Macaulay ring of dimension two, $I$ an $\mathfrak m$-primary ideal such that $I$ is integrally closed and $x \in I$ be a superficial element for $I.$ If $r(I)=2,$ then the Ratliff-Rush filtration of $I$ behaves well mod x.
\end{corollary}

\begin{proof}
   From \cite[Lemma 1.2]{tmj2}, we have $a_{2}(\gr(I))+2 \leq r(I)$. Since $I$ is an integrally closed ideal in $R$, therefore, $b_{1}^{I}(R) \neq -1.$ Thus,  by Proposition \ref{ta} we have, $\omega(I) \leq \max\{1,a_{2}(\gr(I))+1\}$, which implies $\omega(I) \leq  \max\{1,r(I)-1\}$. Since $r(I)=2$, thus $\omega(I) \leq 1$, which implies $\omega(I)=1$. Therefore, the Ratliff-Rush filtration of $I$ behaves well mod $x.$
\end{proof}

Now, as a generalization of Marley's result  \cite[Theorem 2.3]{tjm}, we  establish Theorem \ref{lm}.  Before that we prove the following lemma.

\begin{lemma} \label{fd}
   Let $(R, \mathfrak m)$ be a Cohen-Macaulay local ring of dimension $d \geq 2,$ $I$ an $\mathfrak m$-primary ideal. Suppose that the Ratliff-Rush filtration with respect to $I$ behaves well mod the superficial sequence $x_{1},x_{2},\ldots,x_{d-1} \in I,$ then 
   \begin{center}
       $\rho(I) \geq \rho\left(I/(x_{1})\right) \geq \ldots \geq \rho\left(I/(x_{1}, x_{2}, \ldots, x_{d-1})\right).$
   \end{center}
\end{lemma}

\begin{proof}
    We prove by induction on $d$.
Suppose $d=2.$ Let $x_{1}\in I$ be superficial for $I.$ The Ratliff-Rush filtration with respect to $I$ behaves well mod $x_{1},$ thus, $\frac{\widetilde{I^{n}}+(x_{1})}{(x_{1})}=\frac{\widetilde{I^{n}+(x_{1})}}{(x_{1})}$ for all  $n \geq 1.$ Also $\widetilde{I^{n}}=I^{n}$ for all  $n \geq \rho(I)$. Therefore, $\frac{I^{n}+(x_{1})}{(x_{1})}=\frac{\widetilde{I^{n}+(x_{1})}}{(x_{1})}$ for all  $n \geq \rho(I),$ which implies $\rho\left(I/(x)\right) \leq \rho(I).$
 We assume the lemma for $d-2$ and prove for $d-1.$
    Since the Ratliff-Rush filtration with respect to $I$ behaves well mod the superficial sequence $x_{1},x_{2},\ldots,x_{d-1} \in I,$ then 
    $\pi_{d-2,d-1}$ is surjective. Therefore, $\pi_{d-2,d-1}\left(\widetilde{I^{n}R_{d-2}}\right)=\widetilde{I^{n}R_{d-1}}$ for all $n \geq 1.$ Also $I^{n}R_{d-2}=\widetilde{I^{n}R_{d-2}}$ for all $n \geq \rho\left(I/(x_{1}, x_{2}, \ldots, x_{d-2})\right).$ Thus,  $I^{n}R_{d-1}=\widetilde{I^{n}R_{d-1}}$ for all $n \geq \rho\left(I/(x_{1}, x_{2}, \ldots, x_{d-2})\right).$ Therefore, $\rho\left(I/(x_{1}, x_{2}, \ldots, x_{d-2})\right) \geq \rho\left(I/(x_{1}, x_{2}, \ldots, x_{d-1})\right).$
 \end{proof}

 Now we generalize \cite[Theorem 2.3]{tjm}  by relaxing the assumption on the depth of the associated graded ring.

 \begin{theorem} \label{lm}
      Let $(R, \mathfrak m)$ be a Cohen-Macaulay local ring of dimension $d \geq 2,$ $I$ an $\mathfrak m$-primary ideal. Suppose that the Ratliff-Rush filtration with respect to $I$ behaves well mod a superficial sequence $x_{1},x_{2},\ldots,x_{d-2} \in I.$ Then for all $0 \leq i \leq d$ and for all  $n \geq \rho(I)$

      \begin{center}
          $(-1)^{d-i} \Delta^{i}(P_{I}(n)-H_{I}(n)) \geq 0.$
      \end{center}
 \end{theorem}

 \begin{proof}
   
Note that it suffices to prove the case $i=d.$ Let $g(n)=P_{I}(n)-H_{I}(n).$ 
Suppose $(-1)^{d-i} \Delta^{i}(g(n)) \geq 0$ for all  $n \geq \rho(I)$ and for some $i>0.$ Then  $(-1)^{d-i} \Delta^{i}(g(n))  = \Delta^{1}[(-1)^{d-i} \Delta^{i-1}(g(n))] \geq 0$ for all  $n \geq \rho(I).$ However, $g(n)=0$ for sufficiently large $n$ so  $(-1)^{d-i} \Delta^{i-1}(g(n))=0$  for sufficiently large $n.$ By the property of difference function we have $ (-1)^{d-i} \Delta^{i-1}(g(n))  \leq 0$ for all $n \geq \rho(I).$ 
Therefore, $ (-1)^{d-(i-1)} \Delta^{i-1}(g(n))  \geq 0$ for all  $n \geq \rho(I),$ which proves the theorem for $i-1.$ Hence it is enough to prove  $ \Delta^{d}(P_{I}(n)-H_{I}(n)) \geq 0$ for all  $n \geq \rho(I).$

 We prove this by induction on $d$.

 Case-1: Suppose $d=2.$   Let $x\in I$ be a superficial element for $I.$ Set $S=R/(x)$ and $I'=I/(x),$ then $S$ is a one dimensional Cohen-Macaulay ring.
By Lemma \ref{pushpa} and Proposition \ref{riya}, we have
$H_{I'}(n+1)=H_{I}(n+1)-H_{I}(n)$ for all $n \geq \rho(I),$
    and
$P_{I'}(n+1)=P_{I}(n+1)-P_{I}(n)$ for all $n.$

  Consider 
  \begin{equation*}
\begin{split}
    \Delta^{2}(P_{I}(n)-H_{I}(n)) & =\Delta^{1}[\Delta^{1}(P_{I}(n)-H_{I}(n))]\\
    & = \Delta^{1}[(P_{I}(n+1)-P_{I}(n))-(H_{I}(n+1)-H_{I}(n))]\\
    & = \text{$\Delta^{1}[P_{I'}(n+1)-H_{I'}(n+1)]$  for all $n \geq\rho(I)$}\\
    & \geq \text{0 for all $n \geq\rho(I)$}.
\end{split}
\end{equation*}
Case-2: Suppose $d >2.$ Set $S=R/(x_{1})$ and 
$I_{1}=I/(x_{1})$. Since the Ratliff-Rush filtration of $I$ behaves well mod a superficial sequence $\underline{\mathbf{x}}=x_{1},x_{2},\ldots,x_{d-2} \in I ,$  by Lemma \ref{fd}, we have  $\rho(I) \geq \rho\left(I_{1}\right) \geq \ldots \geq \rho\left(I/(\underline{\mathbf{x}})\right).$ We assume $\Delta^{d-1}(P_{I_{1}}(n)-H_{I_{1}}(n)) \geq 0$ for all $ n \geq \rho(I_{1}).$

  Consider

   \begin{equation*}
\begin{split}
    \Delta^{d}(P_{I}(n)-H_{I}(n)) & =\Delta^{d-1}[\Delta^{1}(P_{I}(n)-H_{I}(n))]\\
    & = \Delta^{d-1}[(P_{I}(n+1)-P_{I}(n))-(H_{I}(n+1)-H_{I}(n))]\\
    & = \text{$\Delta^{d-1}[P_{I_{1}}(n+1)-H_{I_{1}}(n+1)]$  for all $n \geq\rho(I)$}\\
    & \geq \text{0 for all $n \geq\rho(I)-1$}.
\end{split}
 \end{equation*}
  Therefore, we have $ \Delta^{d}(P_{I}(n)-H_{I}(n)) \geq 0$ for all $n \geq\rho(I).$
 \end{proof}

The following corollary partially answers Marley's Question \ref{dc}  in the affirmative.
 \begin{corollary}\label{fg}
     Let $(R, \mathfrak m)$ be a Cohen-Macaulay local ring of dimension $d \geq 2,$ $I$ an $\mathfrak m$-primary ideal. Suppose that the Ratliff-Rush filtration with respect to $I$ behaves well mod a superficial sequence $x_{1},x_{2},\ldots,x_{d-2} \in I.$ If $P_{I}(k)=H_{I}(k)$ for some $k \geq \rho(I)$, then $P_{I}(n)=H_{I}(n)$ for all $n \geq k.$
 \end{corollary}

 \begin{proof}
     By Proposition \ref{lm}, we have 

    \begin{center}
        $ (-1)^{d-1}\Delta^{1}(P_{I}(n)-H_{I}(n)) \geq 0$ for all  $n \geq \rho(I).$
    \end{center}
    Hence
 \begin{center}
        $ (-1)^{d-1}(P_{I}(n+1)-H_{I}(n+1)) \geq (-1)^{d-1}(P_{I}(n)-H_{I}(n))$ for all  $n \geq \rho(I).$
    \end{center}
Also $P_{I}(n)-H_{I}(n)=0$ for sufficiently large $n.$  Thus we get

\begin{center}
     $ 0 \geq (-1)^{d-1}(P_{I}(n)-H_{I}(n)) \geq (-1)^{d-1}(P_{I}(k)-H_{I}(k))$ for all  $n \geq k.$
\end{center}
    Therefore, if $P_{I}(k)=H_{I}(k)$ for some $k \geq \rho(I),$ then $P_{I}(n)=H_{I}(n)$ for all $n \geq k.$
 \end{proof}

Consequently, as a corollary we obtain Marley's result.

\begin{corollary}\label{vv}
    Let $(R, \mathfrak m)$ be a Cohen-Macaulay local ring of dimension $d \geq 2,$ $I$ an $\mathfrak m$-primary ideal such that $\depth \gr(I) \geq d-1.$ Then for all $0 \leq i \leq d$ and for all  $n \geq 1$

      \begin{center}
          $(-1)^{d-i} \Delta^{i}(P_{I}(n)-H_{I}(n)) \geq 0.$
      \end{center}
\end{corollary}

\begin{proof}

Since $\depth \gr(I) \geq d-1,$ from \cite[1.3]{hjls}, we have
\begin{center}
   $\rho(I)= \rho\left(I/(x_{1})\right) = \ldots = \rho\left(I/(x_{1}, x_{2}, \ldots, x_{d-2})\right)=1.$  
\end{center}
Thus, from Theorem \ref{lm},   $\Delta^{d}(P_{I}(n)-H_{I}(n)) \geq 0$ for all $n \geq 1.$
\end{proof}

\section*{Acknowledgement}
The authors sincerely thank the anonymous referee  for meticulous reading the manuscript. The authors also express their gratitude to Prof. T. J. Puthenpurakal, Prof. C. D'Cruz  and Prof. R. Nanduri for their helpful discussions and valuable suggestions. Additionally,  the second author  acknowledges the support of the Government of India through the Prime Minister's Research Fellowship during this work.

\end{document}